\documentclass[12pt,leqno]{amsart}
\usepackage[dvips]{graphics}
\usepackage{amssymb}

\usepackage{hyperref}
\usepackage{xcolor}
\usepackage{enumerate}

\numberwithin{equation}{section}

\newtheorem{thm}{Theorem}[section]

\newtheorem{lem}[thm]{Lemma}
\newtheorem{cor}[thm]{Corollary}
\newtheorem{prop}[thm]{Proposition}

\theoremstyle{remark}
\newtheorem{rem}{Remark}[section]

\newcommand{\tref}[1]{Theorem~\ref{#1}}
\newcommand{\cref}[1]{Korollar~\ref{#1}}

\newcommand{\R}{\mathbb{R}}

\newcommand{\Le}{\mathcal{L}}

\newcommand{\extremal}{maximal }
\newcommand{\Extremal}{Maximal }

\begin{document}
\pagebreak


\title{Lorentz meets Lipschitz}

\author{Christian Lange, Alexander Lytchak, Clemens S\"amann}
\address{Mathematisches Institut, Universit\"at K\"oln, Germany; Department of Mathematics, University of Toronto, Canada.}
\email{\tt clange@math.uni-koeln.de, alytchak@math.uni-koeln.de, clemens.saemann@utoronto.ca}
\date{\today}
\thanks{We thank Benedict Schinnerl for helpful correspondence. C.L.\ and A.L.\ were supported by the DFG-grant SFB/TRR 191 ``Symplectic structures in Geometry, Algebra and Dynamics'', C.S.\ was supported by research grant J4305 of the Austrian Science Fund FWF}


\begin{abstract} We show that \extremal causal curves for a Lipschitz continuous Lorentzian metric admit a $\mathcal{C}^{1,1}$-parametrization 
 and that they solve the geodesic equation in the sense of Filippov in this parametrization. Our proof shows that \extremal causal curves are either everywhere lightlike or everywhere timelike. Furthermore, the proof demonstrates that \extremal causal curves for an $\alpha$-H\"older continuous Lorentzian metric admit a $\mathcal{C}^{1,\frac{\alpha}{4}}$-parametrization.


\end{abstract}


\maketitle


\section{Introduction}

Causality theory is a very active  part of mathematical general relativity and Lorentzian geometry, see the current comprehensive review \cite{Minguzzi:19b}. Recently, in several works \cite{Fathi_Siconolfi:12}, \cite{Chrusciel_Grant:12}, \cite{KSSV:14}, \cite{Minguzzi:15}, \cite{Saemann:16}, \cite{BS:18}, \cite{Minguzzi:19a}, \cite{GKSS:20} the theory was extended to the setting of non-smooth Lorentzian and Lorentz-Finsler manifolds and towards a synthetic theory \cite{Kunzinger_Saemann:2018, AGKS:19, ACS:20, CM:20}, where there might not be a differentiable or manifold structure, analogous to length or Alexandrov spaces in metric geometry. Moreover, issues of low regularity came into focus lately, as they arise in the study of the cosmic censorship conjecture, in particular the (in-)extendibility of spacetimes \cite{Sbierski:18, Galloway_Ling:17, Graf_Ling:2018, GLS:18, GKS:19, Luk_Oh:19a, Luk_Oh:19b, Minguzzi_Suhr:19, Sbierski:20}.
\medskip

The regularity class of $\mathcal{C}^{1,1}$ for the Lorentzian metric turned out to mark a transition between the classical theory and low regularity phenomena. In this regularity, classical results concerning geodesics, causality theory and singularity theorems, remain valid, see \cite{KSS:14, Minguzzi:15, KSSV:14, KSSV:15, KSV:15, GGKS:18, Saemann_Steinbauer}, whereas for lower regularity differences occur \cite{Chrusciel_Grant:12, GKSS:20} and concepts diverge. In particular, the notions of \extremal curves and geodesics (solutions to the geodesic equations) no longer coincide.

If the metric is $\mathcal{C}^1$ then solutions to the geodesic equation still exist but might not be unique. Nevertheless, in the recent work \cite{Graf:20} Graf establishes the Hawking and Penrose singularity theorems in this regularity class and along the way shows that in a (globally hyperbolic) $\mathcal{C}^1$-spacetime any two causally related points are connected by a maximal $\mathcal{C}^2$-geodesic. Building on these results, Schinnerl \cite{Schinnerl:20} proved that in a $\mathcal{C}^1$-spacetime any maximal curve is of $\mathcal{C}^2$-regularity and possesses a parametrization solving the geodesic equation. The latter result of Graf and the one of Schinnerl also follow from our work.

Within  the low-regularity regime Lipschitz continuous Lorentzian metrics play an especially important role as they cover physically most relevant cases: e.g.\ spacetimes where there is a loss of regularity on a hypersurface such as shell-crossing singularities, thin shells of matter, and surface layers, cf.\ \cite{BH:03,SV:17, Lake:17}, or impulsive gravitational waves in the so-called Rosen form, see \cite[Ch.\ 20]{GP:09}, and especially \cite{PSSS:14, PSSS:16}, where the Filippov solution concept was employed to establish $\mathcal{C}^1$-regularity of geodesics.

Below this regularity many fundamental statements in causality theory finally cease to hold, e.g.\ the future of a point need not be open \cite{GKSS:20}, the push-up principle fails and, most surprisingly, there are H\"older spacetimes where the boundary of a lightcone has positive measure (so-called causal bubbles) \cite[Ex.\ 1.11]{Chrusciel_Grant:12}. On the other hand, several fundamental properties of the causality theory have been verified in continuous spacetimes, \cite{Chrusciel_Grant:12, Saemann:16, Minguzzi:19a}. Various other  questions concerning causality theory and properties of \extremal curves in Lipschitz continuous Lorentzian manifolds have been formulated in \cite{Saemann_Steinbauer}. The aim of this article is the resolution of most of  these questions.

\subsection{Statement of results}
All results discussed in the paper are of local nature, thus we will formulate them only for the case of globally hyperbolic Lorentzian manifolds (cf.\ \cite{Saemann:16}).

Given a globally hyperbolic Lorentzian manifold $(M,g)$ with a Lipschitz continuous metric, a causal (Lipschitz continuous) curve $\gamma:[a,b]\to M$ is called \extremal if it maximizes the Lorentzian length among all causal curves connecting its endpoints. Our main result for such curves reads as follows.

\begin{thm} \label{thm: main}
Let $\gamma $ be a  \extremal curve in a Lipschitz continuous Lorentzian metric. Then $\gamma$ admits a parametrization making it  a $\mathcal C^{1,1}$-curve. Moreover, in this parametrization, $\gamma$ is a solution of the geodesic equation in the sense of Filippov.
\end{thm}

If the Lorentzian metric is assumed to be $\mathcal C^1$ then  solutions of the  geodesic equation in the sense of Filippov  are exactly solutions of the geodesic equation in the classical sense and  are automatically $\mathcal C^2$. Thus, for $\mathcal C^1$-regular Lorentzian metrics our results recover \cite[Prop.\ 2.13]{Graf:20} and \cite{Schinnerl:20}.

The proof of Theorem \ref{thm: main} shows that for timelike \extremal curves, their parametrization with constant Lorentzian velocity satisfies the conclusion of Theorem \ref{thm: main},
thus provides a solution of the geodesic equation in the sense of Filippov. 

Due to a recent result by Graf and Ling (\cite[Thm.\ 1.1]{Graf_Ling:2018}) for a Lipschitz continuous Lorentzian metric any \extremal curve is either lightlike or timelike almost everywhere.
We strengthen this result,  see Proposition \ref{prop: final}, and provide an independent proof of it:
\begin{prop} \label{prop: graf}
Let $\gamma  :I\to M$ be a maximal curve for a Lipschitz continuous Lorentzian metric. Then $\gamma$ can be parametrized in a $\mathcal C^{1,1}$ way, and   either $\gamma'(t)$ is lightlike for all $t$, or
$\gamma'(t)$ is timelike for all $t$.
\end{prop}

We mention that the proof does not rely on the push-up principle, established for Lipschitz continuous Lorentzian metrics in \cite{Chrusciel_Grant:12}; in fact, the push-up principle can be easily deduced from Proposition \ref{prop: graf}. 



\subsection{Additional comments} 
We would like to mention a few related statements. First, a direct adaptation of  the proof of Proposition \ref{prop:half_H_cont} reveals the following (probably non-optimal) statement:
\begin{prop}\label{prop_hoelder}
\Extremal curves for an $\alpha$-H\"{o}lder continuous Lorentz\-ian metric admit a $\mathcal C^{1,\frac{\alpha}{4}}$-parametrization.
\end{prop}


The second comment is concerned with the optimality of the $\mathcal C^{1,1}$-regularity in Theorem \ref{thm: main}, whose unclear status was pointed out in \cite{Saemann_Steinbauer}. In a fixed chart, in which the metric is Lipschitz continuous, the \extremal curves do not have to be of class $\mathcal C^2$.  Indeed, we can start with the Minkowski metric  $g_0$ 
on $\R^2$ and pull it back  by a $\mathcal C^{1,1}$-diffeomorphism 
$\phi :\R^2\to \R^2$ such that the preimage $\gamma ^{\ast}:= \phi^{-1} (\gamma)$ of the $y$-axes $\gamma$ is not $\mathcal C^2$.  Then, $\gamma ^{\ast}$ is a \extremal curve in the metric $\phi^{\ast} (g_0)$, and $\gamma^*$ is not $\mathcal C^{2}$.    

The following question also surely has a negative answer, but a proof would require some considerable amount of new analytic ideas, similar to \cite{Kazdan}, \cite{Sabitov} in low regularity: Given a  
	$\mathcal C^{0,1}$-Lorentzian metric $g$ on $U\subset \R^n$, does there exist
some coordinates around any point, in which the metric is still $\mathcal C^{0,1}$ and all \extremal curves are of class $\mathcal C^2$?

The third remark concerns the regularity of solutions of the geodesic equation  in the sense of Filippov, see \cite{Filippov:1988}, \cite{Saemann_Steinbauer} and Section \ref{sec: Filip} for a discussion of the corresponding concepts.
The following general statement shows that the second part of Theorem \ref{thm: main} implies the first and thereby also answers a question posed in \cite{Saemann_Steinbauer}.
\begin{prop} \label{prop: Filippov}
For any  Lorentzian metric $g$  of class $\mathcal C^{0,1}$ on an open domain 
$U\subset \R^n$, any  solution $\gamma :[0,1]\to U$ of the geodesic equation in the sense of Filippov 
 is of class $\mathcal C^{1,1}$.
\end{prop}

Moreover, the $\mathcal C^{1,1}$-norm of $\gamma$ is bounded  in terms of the $\mathcal C^{0,1}$-norm of $g$ and the   Euclidean length of $\gamma$. The proof of this proposition is a direct consequence of a classical result by Filippov \cite[Thm.\ 7.8]{Filippov:1988} and holds true for any pseudo-Riemannian metric $g$ (see also \cite{Steinbauer:14}).

The first half of Theorem \ref{thm: main} is the analog of the corresponding regularity statement for shortest curves in Riemannian metrics, \cite{Yaman}, and its proof  follows the ideas in  \cite{Yaman}.  The proof that \extremal curves in $\mathcal C^{0,1}$-Lorentzian metrics are Filippov-geodesics is an extension of the classical "first variation" argument  and   transfers literally to the Riemannian case (even with some technical issues disappearing). It shows that any shortest curve in a Riemannian $\mathcal C^{0,1}$-metric solves the geodesic equation in the sense of Filippov.

Note, however, that solutions of the geodesic equation in the sense of Filippov are neither uniquely defined nor have any extremality properties, even in the Riemannian case, unless the regularity of the metric $g$ is assumed to be $\mathcal C^{1,1}$,
see \cite{Hartman:50, Hartman_Wintner:51}. In particular, solutions to the geodesic equations need not be \extremal (on any subinterval) nor unique and \extremal curves need not be unique, neither. 


It seems to be a challenging question to understand the geodesic flow  in Lorentzian (or Riemannian) metrics of low regularity, cf.	 \cite{Ambrosio}, \cite{KLP}. Finally, it appears to be possible to use our Theorem \ref{thm: main} to extend the singularity theorem of Hawking to the setting of $\mathcal C^{0,1}$-Lorentzian metrics (this will be explained in a forthcoming article).


\subsection{Structure of the text}
After introducing notations and conventions in Section \ref{sec:not}, we discuss properties of solutions of geodesic equations in the sense of Filippov,  prove Proposition \ref{prop: Filippov} and show that the class of Filippov geodesics is stable under pointwise convergence in Section \ref{sec: Filip}. 
In Section \ref{sec:var} we show that a maximal curve solves the geodesic equation, if it is already known to be 
 timelike, $\mathcal C^{1,1}$  and  parametrized with constant Lorentzian velocity.
 
 In Section \ref{sec: alpha} we state a version of a result of \cite{Calabi-Hartman}, \cite{Yaman} showing that the property of (not) being a $\mathcal C^{1,\alpha}$-curve in a Euclidean space can be characterized in terms of the deviations of the curve from the chords connecting points on the curve.

 In Section \ref{sec: triang} we provide a quantitative version of the triangle inequality in Minkowski space, which to the best of our knowledge seems be a completely new result. 
 
 In the next  sections we fix a chart and 
  show that all \extremal curves in this chart are $\mathcal C^{1,1}$. The basic idea, going  back to the Riemannian situation analyzed in
 \cite{Calabi-Hartman} and \cite{Yaman}, is that if the \extremal curve were too far from 
 a chord, then the chord would have longer Lorentzian length than the \extremal curve.  One of the arising difficulties is that, a priori, the Euclidean chord 
 does not need to be causal.

 In Section \ref{sec: 1-2}  we prove    Proposition \ref{prop_hoelder}, relying on the quantitative triangle inequality. Here we compare the Lorentzian length of the curve and the chord with respect to the Minkowski metric of a fixed point on the chord, and then estimate the error terms arising from the non-constancy of the Lorentzian metric. Also in the Riemannian case, this estimate is too coarse to obtain the optimal $\mathcal C^{1,1}$-regularity, for Lipschitz continuous metrics. However, for the rest of the argument we actually only need that maximal curves are $\mathcal C^{1}$.

 In Section \ref{sec: 1-1} we  consider \extremal curves  whose tangent vectors are timelike everywhere.
 We compare the lengths of the  curve and its chords by estimating the differences of Lorentzian velocities with respect to 
 varying  Minkowski metrics. Also this part follows the Riemannian case in \cite{Yaman}, and supplies us with the $\mathcal C^{1,1}$-regularity for uniformly timelike \extremal curves.

In Section \ref{sec:TL_ex_c} we prove Proposition \ref{prop: graf} and finish the proof of the main theorem for timelike maximal curves.
 
 In the final section \ref{sec:fin} we show that any lightlike \extremal curve can be obtained as a limit of \extremal curves which solve the geodesic equation. By stability, it shows that any \extremal curve solves the geodesic equation in the sense of Filippov and an application of Proposition \ref{prop: Filippov} then finishes the proof of the main theorem.
 
\section{Notation and conventions} \label{sec:not}
 We use the convention that a Lorentzian metric $g$ has signature\\$+---\ldots$ and a vector $v\in T_p M$ is  \emph{causal} if $g_p(v,v)\geq 0$ ($v\neq 0$), \emph{timelike} if $g_p(v,v)>0$, \emph{null} or \emph{lightlike} if $g_p(v,v)=0$ ($v\neq 0$) and \emph{spacelike} if $g_p(v,v)<0$ or $v=0$. Corresponding Lorentzian norms are denoted as $|v|_p:=\sqrt{|g_p(v,v)|}$. A locally Lipschitz continuous curve $\gamma\colon I\rightarrow M$ is \emph{timelike/causal/lightlike} if $\dot\gamma$ is timelike/causal/lightlike almost everywhere.
 
  	On the Euclidean space $\R^n$ we denote by  $\| \cdot \|$ the Euclidean norm.   We denote the standard Minkowski product on $\R^n$ as $\left\langle \cdot, \cdot \right\rangle$,
  	$$\left\langle u,v \right\rangle := u_tv_t-u^T_xv_x \;,$$
  	where we denote for a vector $w\in \R^n$ by $w_t \in \R$ its first  and by 
  	$w_x \in \R^{n-1}$ its last $(n-1)$ coordinates. The Euclidean scalar product will never appear in this text. Otherwise we follow the conventions used in \cite{Chrusciel_Grant:12,GKSS:20}.

  	 A Lorentzian $\mathcal C^{0,1}$-manifold denotes a smooth manifold $M$ with a fixed smooth atlas and a Lorentzian metric $g$ whose coordinates in any chart are Lipschitz continuous.
  	
  	Most of the time we work locally on an open set $U\subset \R^n$. 
  	In such a chart we denote the  Lipschitz constant  of $g$ as $L$:
 $$|g_x(v,w)-g_y(v,w)| \leq L\cdot \|x-y\| \;,$$
for all unit vectors $v,w\in \R^n$ and all $x,y\in U$.


By making the chart smaller and the Lipschitz constant larger, if needed, we can always assume that for any point $x\in U$ the chart  $U$ can be changed, so that $g_x$ coincides with the standard Minkowski product   $\left\langle \cdot, \cdot \right\rangle$,  (while keeping the Lipschitz constant $L$ for $g$).

 Making  the chart $U$ smaller  we can   assume that the first coordinate vector $T$ is future directed timelike. More precisely,  we may assume  that the first coordinate $t(u) :=u_t$ as function on $\R^n$ grows with velocity at least $\frac 1 2$ on any future directed causal curve $\gamma:I\to U$ parametrized by Euclidean arclength, i.e.
 $$(t\circ \gamma ) ' \geq \frac 1 2 $$
almost everywhere. In particular, any causal curve $\gamma$ in $U$ parametrized by arclength is a bilipschitz curve with bilipschitz constant $2$. Our Lorentzian  manifold $M$ and therefore also all charts will always be assumed globally hyperbolic (cf.\ the proof of \cite[Thm.\ 2.2]{Saemann_Steinbauer}).


 
 %


 By $\mathcal L(\gamma) =\mathcal L_g (\gamma) $ we denote the Lorentzian length of a causal curve $\gamma :I\to M$:
 $$\mathcal L(\gamma) =\int _I   \sqrt{g(\gamma '(t), \gamma '(t))} \; dt \;.$$
 
 A causal curve $\gamma \colon I\rightarrow M$ is called \extremal if its Lorentz\-ian length is maximal among all causal curves with the same endpoints. Any subcurve of any \extremal curve is maximal.
 
\section{Filippov geodesics} \label{sec: Filip}
\subsection{Filippov geodesics, their regularity and stability}\label{sub:Filippov_regularity}
Let $M$ be a Lorentzian manifold with Lorentzian metric $g$ of class $C^{0,1}$.
Then there exists a unique Levi-Civita connection
$\nabla ^g$ on $M$. This connection assigns to any pair of locally Lipschitz vector fields $X,Y$ a locally bounded vector field $\nabla ^g_X Y$ so that 
the usual rules (torsion freeness and metric property) are satisfied, see, for instance, \cite{Honda}.

In any chart $U\subset \R^n$ the Levi--Civita connection differs from the directional derivative $DY(X)$ by a tensor $\Gamma$, called the Christoffel symbol, which is a bounded symmetric $(1,2)$-tensor on $U$.

The usual Koszul formula provides a formula for  $\Gamma$ via the partial derivatives of $g$, which exist almost everywhere on $U$. More precisely, the value of $\Gamma$ is defined in all points in which $g$ is differentiable.


Let $U$ be a chart as in Section \ref{sec:not}.  We say that a curve 
$\gamma :I\to U$ is a  (Filippov-) geodesic if $\gamma'$ exists almost everywhere, is absolutely continuous, and $\gamma$ is a solution of the  differential equation 
$$\gamma''(t) + \Gamma_{\gamma(t)} (\gamma'(t), \gamma'(t)) =0  \;$$
in the sense of differential inclusions of Filippov \cite{Filippov:1988}, \cite{Saemann_Steinbauer}. This is the case if for almost every $t \in I$ the point $-\gamma''(t)$ is contained in the essential convex hull
\[
			\hat\Gamma_{\gamma(t)}(\gamma'(t), \gamma'(t)) := \bigcap_{\delta>0} \bigcap_{\mu(N)=0} 
			 \mathcal K (\delta, N),
%
\]
where  $\mu$ is the Lebesgue measure on $\R^n\times \R^ n$ and 
$\mathcal K(\delta, N)$   is the closed convex hull
\begin{align*}
 \mathcal K(\delta, N) &:=\mathrm {co} \{\Gamma _x (w,w) \; :  \; (x,w)\in (U \times \R^n) \setminus N \; ; \;\\
 &\qquad\quad\|(x,w)-(\gamma(t), \gamma'(t))\|<\delta \} \,.
\end{align*}

Due to the continuity of $\Gamma _x(v,v)$ in the second variable $v$ for almost all $x\in U$, it suffices in the definition of the essential convex hull
$\hat \Gamma_ {\gamma} (\gamma', \gamma')$  to consider  subsets $N$ with $\mu (N)=0$ of the form $N=N_0\times \R^n$.

The notion is invariant under smooth changes of coordinates and therefore, we can unambiguously talk about Filippov-geodesics in the Lorentzian manifold $M$, by requiring that the intersection of the curve with any chart is a Filippov-geodesic in this chart.

\begin{rem}
Filippov-geodesics are invariant under $\mathcal C^2$-changes of coordinates. However, it is not clear to us if Filippov geodesics are invariant under $\mathcal C^{1,1}$-changes of coordinates, the most natural class in the context of $\mathcal C^{0,1}$-metrics.
\end{rem}

We can now provide:
\begin{proof}[Proof of Proposition \ref{prop: Filippov}]
The statement is local, thus we may assume that the neighborhood $U$ is small enough as in Section 2, in particular, $g$ is globally Lipschitz continuous on $U$.
Let $\gamma: [0,1] \to U$ be a solution of the geodesic equation in the sense of Filippov.  Then $\gamma'$ is absolutely continuous, in particular, 
$\|\gamma ' (t)\|$ is bounded by some constant $C_1$.

For almost every $x\in U$ the Christoffel symbol $\Gamma_x  $ has a norm (as a bilinear map) bounded
  by a constant $C_2$ depending  on the Lipschitz norm of $g$.  Hence, for almost all $x\in U$ and  all $t\in [0,1]$ we have  
  $$\|\Gamma _x (\gamma '(t), \gamma ' (t))\|\leq C_1^2\cdot C_2 \;.$$
  Thus the same bound is true for any convex hull of vectors of the form
  $\Gamma _x (\gamma '(t), \gamma ' (t)) $. The definition of being a solution of the geodesic equation now implies for almost all $t\in [0,1]$
$$\|\gamma'' (t) \| \leq C_2\cdot C_1 ^2 \;.$$
Thus $\gamma '(t)$ is Lipschitz continuous and $\gamma $ is of class $\mathcal C^{1,1}$.
\end{proof}

As the proof shows,  the $\mathcal C^{1,1}$-norm of $\gamma$ depends only
on the Lipschitz constant of $g$ and an upper bound on the velocity $\|\gamma '\|$ of $\gamma$.
Thus, due to the next observation, the  $\mathcal C^{1,1}$-norm is even controlled by the length:

\begin{lem} \label{lem: fil}
 Let  $\gamma : [a,b] \to U$ be a solution of the geodesic equation in the sense of Filippov.
 Then, for all $t\in [a,b]$, we have
 $$\|\gamma' (t)\| \leq  \frac {e^{Cr}-1} {C}   \cdot \frac 1 {b-a} \;,$$
 where $C$ is some constant depending only on  $U$ and where $r$ denotes the Euclidean length of $\gamma$.  
\end{lem}

\begin{proof}
By the form of the geodesic equation, the $\mathcal C^{1,1}$-curve $\gamma$ satisfies for almost all $t\in [a,b]$  the differential inequality $$\|\gamma''(t)\| \leq C\cdot \|\gamma '(t)\|^2 \;,$$
where the constant  $C>0$ depends only on the Lipschitz constant of the Lorentzian metric $g$.	

Thus,  the Lipschitz continuous non-negative function  $s(t):= \|\gamma '(t)\|$ satisfies almost everywhere
$$|s'(t)| \leq C\cdot s^2  (t)\;.$$
If $s(t)\neq 0$, this is equivalent  to 
$$\left|\left(\frac 1 {s(t)} \right) '\right| \leq C \;.$$
Thus, on the open set of points $t$ at which $s$ is positive, the functions
$\frac 1 {s(t)}$ is $C$-Lipschitz continuous.  By continuity, either $s$ is constantly equal to $0$
or $s(t)>0$, for all $t$.  

In the first case, $\gamma$ is constant and the conclusion clearly holds.
In the second case, the function $\frac 1 {s(t)}$ is $C$-Lipschitz continuous on $[a,b]$.
Let, in the second case,  $s$ assume its maximum $s_0$ on $[a,b]$ at the point $t_0$.
Setting $\epsilon:=\frac 1 {s_0}$, we get 
$$\frac 1 {s(t)} \leq \epsilon + C\cdot |t-t_0|\;, $$
$$s(t) \geq \frac 1 {\epsilon + C\cdot |t-t_0|}\;.$$
Thus, we can estimate the Euclidean length of $\gamma$ as
$$r= \int _a ^b s(t) \, dt  \geq \int_ a^b \frac 1 {\epsilon + C\cdot |t-t_0|} \,dt \geq \int_ a^b \frac 1 {\epsilon + C\cdot (t-a)} \,dt  =$$
$$=\frac 1 C \cdot  (\log (C(b-a) +\epsilon) -\log (\epsilon))=\frac 1 C \cdot \log (1+ \frac { C\cdot (b-a)} {\epsilon}) \;.$$
Therefore,
$$e^{Cr} -1\geq   \frac { C\cdot (b-a)} {\epsilon}$$
Recalling that $\frac  1 {\epsilon}$ is the maximum of $\|\gamma'\|$ completes the proof.
 %
%
\end{proof}



As a consequence we derive:
\begin{cor}\label{cor_fil_lim}
	Let $\gamma_i :[a,b]\to U$ be a sequence of  causal curves converging pointwise to  $\gamma:[a,b]\to U$.  If all curves  $\gamma _i$ are Filippov-geodesics then so is $\gamma$.
%
\end{cor}

\begin{proof}
	 By  our assumption on $U$, all causal curves in $U$ have Euclidean length bounded by a uniform constant $r_0$.  By Lemma \ref{lem: fil}, this implies that all $\gamma_i$ have uniformly  bounded first derivatives. 
	 As has been shown in Proposition \ref{prop: Filippov}, this implies that the curves $\gamma _i$ are uniformly $\mathcal C^{1,1}$. Then, the pointwise convergence implies that $\gamma _i$ converges uniformly to $\gamma$ and 
	$\gamma_i'$ converges uniformly to $\gamma'$. Therefore we obtain the conclusion by  applying stability theorems  for solutions of differential inclusions \cite[Cor.\ 1, Ch.\ 2, \S 7, p.\ 77]{Filippov:1988}.
\end{proof}


  \section{First variational formula}\label{sec:var}
    In the remainder of this section we are going to prove the following statement. The proof is  the standard variational argument, known as ``first variation formula of length" or the du Bois-Reymond trick. The only non-classical point in the proof is to carry out all arguments not for a single variation but for a ``full-dimensional" family of variations. This  allows us to use the properties of the Levi-Civita connection, which in our setting are valid only almost everywhere.
  
  \begin{prop} \label{prop: geodeq}
  	Let $g$ be a Lipschitz continuous Lorentzian metric on a domain $U\subset \R^n$. Let $\gamma :[0,1]\to U$ be a $\mathcal C^{1,1}$-timelike \extremal curve with constant Lorentzian speed $|\gamma'(t)|_g$. Then $\gamma$ solves the geodesic equation in the sense of Filippov. 
  \end{prop}



  Assume that the statement is wrong. We then find a Lebesgue point $t_0\in I$ of the $L^{\infty}$-function $\gamma'' :I\to U$, a negligible set $N\subset U$ (containing the set of all points where $g$ is not differentiable) and some $\delta >0$
  such that the following holds true.  The closed convex hull $K$ of the
  set of all vectors of the form $\Gamma_{x } (v,v)$ with $(x,v)\in B_\delta(\gamma (t_0),\gamma'(t_0)) \setminus (N \times \R^n)$ does not contain $-\gamma''(t_0)$. 
	Thus we find a unit vector $h$ and some $\epsilon >0$ such that \begin{equation}\label{eq:separation}
	g_{x_0}( h,\gamma''(t_0)+ k ) >\epsilon
\end{equation}
  for all $k\in K$, where $x_0=\gamma(t_0)$. In fact, (\ref{eq:separation}) holds for all $h$ in a small nonempty open set $O$. Shrinking $O$ we can assume that all $h\in O$ are linearly independent of $\gamma' (t_0)$.
  
  Reparametrizing $\gamma$ and changing coordinates, we may assume  that $t_0=0$, $\gamma (0)=0$ and that $\gamma$ is parametrized on an interval $[-r,r]$ for some small $r$. We further assume that $g_0$ is the standard Minkowski product on $\R^n$.	By continuity of $\gamma '$  we may further assume that the inequality (\ref{eq:separation})
holds for all  $k$ of the form $\Gamma _x (\gamma '(t), \gamma'(t))$, where $t\in[-r,r]$ and $x$ in $B_\delta(0) \setminus N$ are arbitrary.

  
  

	In the following we can and will assume that $\sup_{u_1,u_2\in U}\|u_1-u_2\|<1$.  Let $H$ be a hyperplane containing $h$ and transversal to $\gamma' (0)$. Since $0$ is a Lebesgue point of $\gamma''$, we have for all small $r_0>0$
	\begin{equation}\label{eq:Lebesgue_point}
			\int _{-r_0} ^{r_0} \left\| \gamma''(t)-\gamma''(0) \right\| dt \leq \frac{\epsilon}{8L}\cdot r_0,
	\end{equation}
where $L$ is the Lipschitz constant of $g$.

We fix such an $r_0$ and choose a smooth function $f\colon[-r_0,r_0]\rightarrow [0,1]$ vanishing only in $\pm r_0$ with $f(t)\geq 1/2$ on $[-r_0/2,r_0/2]$. For a small neighborhood $W$ of $0$ in $H$, we consider the map $F:W\times [-r_0,r_0]\to U$ given by $$F(w,t)=-f(t)\cdot w+\gamma (t)\;.$$  
  
  For any $w\in W$, the curve $\gamma _w (t) :=F(w,t)$ has the same endpoints as
  $\gamma$.  By assumption $|\gamma'(t)|_g$ is a positive constant. Hence, if $W$ is small enough, then all curves $\gamma_w$ are timelike. 
	
	The map $F$ is a $\mathcal{C}^{1,1}$-diffeomorphism onto the image outside the boundary points $\gamma (\pm r_0)$. For sufficiently small $W$ the map 
  $$w\to \Le_w :=\Le(\gamma _w) :=  \int _{-r_0} ^{r_0} \sqrt{g(\gamma _w'(t), \gamma _w'(t))} \; dt$$
  is Lipschitz continuous.   

  We can follow the classical computation for the first derivative of the length, cf.\ e.g.\ 
  \cite[Prop.\ 10.2]{ONe:83} or \cite[Ch.\ 3]{Grove}. We see that for almost all small $w\in W$ the derivative 
  $\frac \partial {\partial s}|_{s=\rho} \Le_{sw}$ exists for almost all small $\rho$ and can be computed by the usual formula 
  \[
		\frac \partial {\partial s}\biggr\rvert_{s=\rho} \Le_{sw} = \int _{-r_0} ^{r_0} f(t) g(w,\tilde v_{\rho w}'(t)+\Gamma_{\gamma_{\rho w}(t)}(\tilde v_{\rho w}(t),\gamma_{\rho w}'(t))) dt \]		
	where $\tilde v_{\rho w}(t) :=\frac{\gamma_{\rho w}'(t)}{\left|\gamma_{\rho w}'(t)\right|}$.
	
	We claim that for almost all $w\in O$ this expression is positive for almost all small $\rho>0$ (and we only consider such $\rho$ with $\rho\cdot O\subset W$). Indeed for such $w$ we have
\begin{equation*}\label{eq:crt_7}
	\begin{split}
		\frac \partial {\partial s}\biggr\rvert_{s=\rho} \Le_{sw} \geq &\int _{-r_0} ^{r_0}  f(t) g(w,\tilde v_{0}'(0)+\Gamma_{v_{\rho w}(t)}(\tilde v_{\rho w}(t),\gamma_{\rho w}'(t))) dt \\
		 &-L\bigl(\int _{-r_0} ^{r_0}  \left\|\tilde v_{0}'(t)-\tilde v_{0}'(0)\right\| dt -\int _{-r_0} ^{r_0}  \left\|\tilde v_{\rho w}'(t)-\tilde v_{0}'(t) \right\| dt\bigr)\,. \\
			\end{split}
 \end{equation*}
By continuity of $\tilde v_{\rho w}'(t)$ in $\rho$ (at times of existence) and the constant Lorentzian speed $|v_t|$ of $\gamma$, the integrand of the last term can be made smaller than $\epsilon/16|v_t|$ by choosing $\rho$ small enough. The second term can be estimated with (\ref{eq:Lebesgue_point}). Hence
\[
		\frac \partial {\partial s}\biggr\rvert_{s=\rho} \Le_{sw} \geq \frac{1}{2|v_t|} \int _{-r_0/2} ^{r_0/2}  g\left(w,\tilde v_{0}'(0)+\Gamma_{v_{\rho w}(t)}(\gamma_{\rho w}'(t),\gamma_{\rho w}'(t))\right) dt - C\frac{ r_0 \epsilon}{|v_t|}
\]
with $C<\frac 1 2$. By Lipschitz continuity of $g$ and boundedness of $\Gamma$ the same estimate holds for $g_0$ instead of $g$, if we choose $r_0$ much smaller than $\epsilon$. By continuity in $\rho$ and our choice of $w$ we obtain
\[
		\frac \partial {\partial s}\biggr\rvert_{s=\rho} \Le_{sw} \geq   \frac{ r_0 \epsilon}{2|v_t|} - C\frac{ r_0 \epsilon}{|v_t|} > 0
\]
 for almost all sufficiently small $\rho>0$. Thus, $\mathcal L_{\rho w} =\mathcal L(\gamma _{\rho w}) >\mathcal L(\gamma)$ for almost all sufficiently small $\rho$, in contradiction
to the maximality of $\gamma$. This finishes the proof of Proposition \ref{prop: geodeq}.





\section{$\mathcal C^{1,\alpha}$ curves} \label{sec: alpha}

The following lemma is a variant of \cite[Lem.\ 2.1]{Calabi-Hartman}, modified for our purposes.

\begin{lem}\label{lem:geo_diff_crit}
Let $\gamma :[a,b]\to \R^n$
be a bilipschitz curve parametrized by arclength.  Let $0<\alpha \leq 1$ be fixed.
Then the following are equivalent:
\begin{enumerate}
\item The curve $\gamma$ is $\mathcal C^{1,\alpha}$.

\item There exists some $C\geq 1$ such that for all $0\leq h\leq \frac 1 C$ and all
intervals $[t,t+h]\subset [a,b]$ the image of the restriction $\gamma ([t,t+h])$ is contained in the tube $B_{C\cdot h^{1+\alpha}} ([\gamma (t),\gamma (t+h)])$ of radius $C\cdot h^{1+\alpha}$ around the linear segment between $\gamma (t)$ and $\gamma (t+h)$.
\end{enumerate} 
\end{lem}

\begin{proof}
 If $\gamma$ is $\mathcal C^{1,\alpha}$, then there exists some $C>1$ such that for all small
 $h$ and all $r\in [t,t+h]$ we have (see \cite[Lemma 2.1]{Yaman})
 $$\|h\cdot  \gamma '(r) - (\gamma (t) -\gamma (t-h))\| \leq C\cdot h^{1+\alpha} \;.$$	
	Dividing by $h$ and integrating from $t $ to $r$ we deduce
$$\gamma (r)=  \gamma (t)+ \frac{r-t} {h} \cdot  (\gamma (t+h) -\gamma (t)) + q(r)\;,$$
where $\|q(r)\|\leq C\cdot h^{1+\alpha}$. Since $q(r)$ bounds the distance to the segment between $\gamma (t)$ and $\gamma (t+h)$, this proves that (1) implies (2).

Assume (2) and let $L$ denote the biLipschitz constant of $\gamma$, i.e. $|t_2-t_1|\leq L d(\gamma(t_1),\gamma(t_2))$.  
Set $X=\gamma ([a,b])$ and $\epsilon:= h /L$.
Consider arbitrary points $x_1,x_2\in X$ with $\|x_1-x_2\| \leq \epsilon$.

Considering the projection of the part $X'$ of $\gamma$ between $x_1$ and $x_2$ onto the segment  $[x_1,x_2]$ we find a point $p$ in $X'$  which is projected onto the midpoint $m$ between $x_1$ and $x_2$.  By assumption 
$d(p,m)\leq  C\cdot d^{1+\alpha} (x_1,x_2)$.

By Pythagoras, we deduce
that for $i=1,2$ 
$$d(x_i,p)\leq \frac 1 2 d(x_1,x_2) \cdot (1+ 4C^2 d^{2\alpha} (x_1,x_2))\;$$
By \cite[Proposition 1.1]{Ly-conv} and \cite[Theorem 1.2]{Ly-conv} there exists some $\bar \epsilon >0$ such that all pairs of points in $X$ at distance less than $\bar \epsilon$ are connected in $X$ by a $\mathcal C^{1,\alpha}$-curve parametrized by arclength. Since $\gamma$ is injective, this curve must coincide with $\gamma$, thus finishing the proof.
\end{proof}


\section{Quantitative triangle inequality} \label{sec: triang}
Recall that vectors $u,v \in C_+$ in the positive causal cone \[C_+ = \{ v=(v_t,v_x) \in \R \times \R^n=\R^{n+1} \mid \left\langle v ,v\right\rangle \geq 0, \; v_t > 0 \}\] of the 	
	standard Minkowski space $(\R^{n+1}, \left\langle \cdot,\cdot \right\rangle)$ with Minkowski product $\left\langle u,v \right\rangle = u_tv_t-u^T_xv_x$ and norm $\left|v\right| = \sqrt{\left|\left\langle v,v\right\rangle\right|}$ satisfy the reversed triangle inequality
\begin{equation}\label{eq:reverse_triangle}
			\left|u+v \right|\geq \left|u\right| + \left|v\right|
\end{equation}
as well as the reversed Cauchy-Schwarz inequality $\left\langle u,v \right\rangle \geq \left| u \right| \left|v\right|$, see e.g.\ \cite[Prop.\ 5.30]{ONe:83}. Here we describe a quantitative version of the reversed triangle inequality:
\begin{lem}\label{lem:cont_rev_tria_ineq} For $A=\frac{1}{10}$ and for all $u,v \in C_+$, the inequality
\begin{equation}\label{eq:controlled_reverse_triangle}
			\left|u+v \right|^2 \geq \left|u+v \right| \left(\left|u\right| + \left|v\right|\right) + A \cdot D^2
\end{equation}
holds, where $D$ is the Euclidean distance from $u$ and $v$ to the Euclidean line spanned by $u+v$.
\end{lem}
\begin{proof} By (\ref{eq:reverse_triangle}) we can suppose that $u$ and $v$ are linearly independent.
By homogeneity and by the $\mathrm{O}(n)$-invariance in the spatial part of $\R^{n+1}$ we can assume that $n=2$ and that 
\begin{equation*}\label{eq:crt_2}
		u=(t,x,y), \ \  v=(1,z,0) 
\end{equation*}
with $0\leq t\leq 1 \; , \;  0\leq z\leq 1$ and $x^2+y^2\leq t^2$.
Using the Cauchy-Schwarz and the reversed triangle inequality we compute
\
\begin{equation*}
\begin{split}
\left|u+v \right|- \left|u\right| - \left|v\right| & = \frac{\left|u+v \right|^2- \left(\left|u\right| + \left|v\right|\right)^2}{\left|u+v \right|+ \left|u\right| + \left|v\right|} \\
 & \geq  \frac{\left\langle u,v \right\rangle-\left|u\right|\left|v\right|}{\left|u+v \right|} \\
 & \geq  \frac{\left\langle u,v\right\rangle^2-\left|u\right|^2\left|v\right|^2}{2\left|u+v \right|\left\langle u,v\right\rangle} \\
 & =  \frac{(t-xz)^2-(t^2-x^2-y^2)(1-z^2)}{2\left|u+v \right|\left\langle u,v \right\rangle} \\
 & =  \frac{(tz-x)^2+y^2(1-z^2)}{2\left|u+v \right|\left\langle u,v \right\rangle}. \\
\end{split}
\end{equation*}
We also observe that for such $u$ and $v$
\begin{equation*}\label{eq:crt_4}
			D^2=\frac{\left\|u\times v\right\|^2}{\left\|u+v\right\|^2}=\frac{(tz-x)^2+y^2(1+z^2)}{\left\|u+v\right\|^2}   \geq  (tz-x)^2+y^2(1+z^2),
\end{equation*}
because  $\left\|u+v\right\|\geq 1$.
 We deduce that 
\begin{equation}\label{eq:crt_5}
	{\frac {|u+v|^2- |u+v|(|u|+|v|) } {D^2}}  \geq 
	{\frac{(tz-x)^2+y^2(1-z^2)}{(tz-x)^2+y^2(1+z^2)}} \cdot {\frac{1}{2(t-zx)}}
\end{equation}
For fixed $t,z,x$, the term on the right is monotone decreasing in $y^2$. Thus, we may assume that $y^2$ is as large as possible,  $y^2=t^2-x^2$.
Then 
$$(tz-x)^2 +y^2 (1-z^2)= (t-zx)^2$$
and  the right hand side of \eqref{eq:crt_5} is at least
$$\frac {t-zx} {(tz-x)^2+(t^2-x^2)(1+z^2)} \geq  \frac{1}{2} \cdot \frac{t-zx} {\left|tz-x \right| + (t-x) \cdot 2\cdot 2 } \;. $$
 Since $t-zx\geq  \left|tz-x \right| $ and $t-zx \geq t-x$ we arrive at  
 $$	\frac {|u+v|^2- |u+v|(|u|+|v|) } {D^2}  \geq \frac {t-zx} {10\cdot  (t-zx)} =\frac{1}{10} \;.$$
 This finishes the proof.
\end{proof}
 
We can now easily derive the following conclusion: 	
 	
 \begin{cor} \label{cor:geom}
  For any compact set $\mathcal S$ of Lorentzian bilinear forms $\lambda :\R^n \times \R^n \to \R^n$ there exists a constant $A=A(S)$ such that for any 
  $\lambda \in S$ the following holds true:
  
  If $\gamma :[a,b] \to \R^n$ is an $\lambda$-timelike  curve, if $w$ is the vector $\gamma (b)-\gamma (a)$ and 
  if $D$ denotes the maximal Euclidean distance of a point on $\gamma$ from the Euclidean line through $\gamma (a) $ and $\gamma (b)$, then
\begin{equation}\label{eq:controlled_reverse_triangle+}
\left|w \right| _{\lambda} \geq \mathcal L_{\lambda} (\gamma)   + \frac {A \cdot D^2} {\left|w \right| _{\lambda}}
\end{equation}
 \end{cor}

\begin{proof}
Any $\lambda$ in $\mathcal S$ can be brought by a bounded linear transformation to the standard Minkowski product. By compactness, all these transformations change the Euclidean norm  (and thus the distance $D$ appearing in the formula) by a uniformly bounded factor. Thus it suffices to prove the statement
for the standard Minkowski product $\left\langle \cdot , \cdot \right\rangle$.

We first observe that in Minkowski space all chords of a timelike curve are timelike as well.
Now, by homogeneity, we can assume $\gamma (a)=0$. We consider the point $u$ on $\gamma$ with maximal distance  to the line through $0$ and $\gamma (b)=w$.
Then for $v=w-u$ we can apply Lemma \ref{lem:cont_rev_tria_ineq}, to deduce
$$|w| \geq |u|+|v| + \frac {AD^2} {|w|}\;.$$
Since straight lines in Minkowski space maximize Lorentzian length, we have $\mathcal L(\gamma) \leq |u|+|v|$, and this finishes the proof.
\end{proof}

\section{First step to regularity} \label{sec: 1-2}
We now embark on the proof of Theorem \ref{thm: main}. The statement is local
and we will  restrict from now on to a chart $U$ as in Section \ref{sec:not}. 

We will need to compare Lorentzian lengths with respect to different Lorentzian metrics and will rely on the following observation, a direct analog
of \cite[Lemma~3.1]{Yaman}. The proof is obtained by integration.

\begin{lem}\label{lem:yaman} Let $\varepsilon>0$ be sufficiently small and let $g_1,g_2$ be Lorentzian metrics on $U$ such that $|g_1(v,v)-g_2(v,v)|< \varepsilon$ for all points in $U$ and all unit vectors $v$. Let $\gamma: [0,\delta] \rightarrow U$ be a curve parametrized by Euclidean arclength which is causal with respect to $g_1$ and $g_2$. Then the Lorentzian lengths of $\gamma$ with respect to $g_1$ and $g_2$ can be compared as $|\Le_1(\gamma)-\Le_2(\gamma)|\leq\delta \cdot \sqrt{\varepsilon}$. 
	
	
\end{lem}

In this section we are going to prove

\begin{prop}\label{prop:half_H_cont}
	Any \extremal curve $\gamma :I\to U$ in $U$ parametrized by Euclidean arclength is $\mathcal C^{1,\frac 1 4}$.
\end{prop}
\begin{proof}
 Assume the contrary. By Lemma \ref{lem:geo_diff_crit} we find for any arbitrarily large 
$C>0$, an arbitrarily small $h>0$, some  subinterval $[t,t+h]\subset I$, and some point $q$ on $\gamma ([t,t+h])$ such that 
\begin{equation}\label{eq:7}
D \geq C h^{1+\frac{1}{4}},
\end{equation}
where $D$ denotes the Euclidean distance from $q$ to the Euclidean line between
$\gamma (t)$ and $\gamma (t+h)$.

The vector $$u:=\frac {\gamma (t+h)-\gamma (t)} h$$
satisfies $\|u\|\leq 1$ since $\gamma$ is $1$-Lipschitz.

We consider the restriction $\gamma_{t,h}=\gamma_{|[t,t+h]}$ and  the curve 
  $\tilde \gamma =\tilde \gamma_{t,h}  :[0,h]\to U$ given by 
 $\tilde \gamma (s):= \gamma(t)+s\cdot u$.  The curve $\tilde\gamma$ has the same endpoints as the \extremal curve $\gamma_{t,h}$.  A contradiction  proving 
  Proposition \ref{prop:half_H_cont} will be provided by showing that for $C\geq C(L)$ large enough the curve $\tilde \gamma$ is timelike and has larger Lorentzian length than $\gamma_{t,h}$.
 
 We may assume that $t=0$, that $\gamma (t)=0\in U$
and that $g_0$ is the standard Minkowski product  $g_0= \left\langle \cdot, \cdot \right\rangle$,  see Section \ref{sec:not}. We set $\gamma_{h}:=\gamma_{t,h}$. 
We consider the auxiliary Lorentzian products on $\R^n$  given by     
$$g ^{h}  (v,w) =\left\langle v,w\right\rangle + 4\cdot L\cdot h \cdot  v_t \cdot w_t \; ,$$ 
 where as before $v_t$ denotes the first coordinate of $v\in \R^n$.

 Then, for all $x$ in $\R^n$ with $\|x\| \leq h$ and for all unit vectors $v\in \R^n$ we have
 \begin{enumerate}
 \item $|g_x(v,v)-g^h(v,v)|\leq 5 L h$, and
 \item if $g_x (v,v)> 0$ then $g^{h} (v,v) >0$,
 \end{enumerate} 
 compare \cite[Thm.\ 1.15]{Chrusciel_Grant:12}. In particular, any $g$-causal curve in the ball $B_h(0)\subset \R^n$ is $g^h$-causal.

Hence, the curve $\gamma _{h}$ is timelike with respect to the Minkowski product
 $g^{h}$. Thus also $u$ and the straight segment  $\tilde \gamma$ are $g^{h}$-timelike. 
 
 Note that all our auxiliary Lorentzian products $g^h$ are contained in a compact set of Lorentzian products, fixed independently of $h$. Thus, we find a constant $A>0$ such that the conclusion of Corollary \ref{cor:geom} is valid for all such $g^h$, and so we deduce:
 \begin{equation} \label{eq:diff}
 \mathcal L _{g^{h}} (\tilde \gamma)  = |\gamma (h)-\gamma (0)| _{g^{h}} =h \cdot |u|_{g^{h} } \geq \mathcal L _{g^{h}} (\gamma)   +  \frac {A\cdot D^2} {h\cdot |u| _{g^{h}}}\,.
 \end{equation}
 Inserting $D\geq C\cdot h^{1+ \frac  1 4}$, we obtain
 $$|u|_{g^{h}} \geq \frac {\sqrt A \cdot D }  {h}  \geq \sqrt A\cdot C\cdot h^{\frac 1 4}\;.$$
 Thus, for all $x\in B_{h} (0)$  we have
 $$g_x (u,u) \geq AC^2 \cdot h^{\frac 1 2} -  5 L  \cdot h \geq \frac 1 2 AC^2 \cdot h^{\frac 1 2}$$
 for all sufficiently small $h$. Therefore the curve $\tilde \gamma$ is $g$-timelike. Moreover, we have
\begin{equation*}
\begin{split}
\Le(\tilde \gamma) -\Le(\gamma_{h}) =&(\Le_{g^{h}} (\tilde \gamma ) -\Le_{g^{h}} (\gamma_{h}))\\ &+\underbrace{(\Le(\tilde \gamma) -\Le_{g^{h}} (\tilde \gamma))+
(\Le_{g^{h}} (\gamma_{h} ) -\Le(\gamma_{h}))}_{=:R}.
\end{split}
\end{equation*}
By Lemma \ref{lem:yaman}, we deduce that the absolute value of $R$ is at most 
$2\sqrt{5L}h^{\frac 3 2}$. Since $\gamma$ is \extremal, $\Le(\tilde \gamma)-\Le(\gamma_{h}) \leq 0$. Applying \eqref{eq:diff}, we find a constant $C_1>0$ such that
$$\frac {AD^2} {h |u|_{g^{h}} }\leq C_1 \cdot h^{\frac 3 2}$$
$$AD^2 \leq C_1 \cdot h^{\frac 5 2} \cdot |u|_{g^h}$$ 
Since $|u|_{g^h}\leq 2$ we deduce that
$D\leq C_2 \cdot h^{ 1 + \frac 1 4}$ for some constant $C_2$ (independent of $\gamma$, $t$ and $h$), and thus a contradiction, if $C$ has been chosen large enough. This finishes the proof.
\end{proof}


Easily adapting the arguments we provide
\begin{proof}[Proof of Proposition \ref{prop_hoelder}]
 The proof follows literally as above by redefining $g^h$ as
 $$g ^{h}  (v,w) =\left\langle v,w\right\rangle + 4\cdot L\cdot h^\alpha \cdot  v_t \cdot w_t \; ,$$
 and replacing $\frac{1}{4}$ by  $\frac{\alpha}{4}$ in \eqref{eq:7}.
\end{proof}

\section{Regularity away from lightlike vectors} \label{sec: 1-1}

	In this section we
	prove,  again following \cite{Yaman}:

\begin{thm}\label{thm:extremal_timelike}
Let $\gamma: I\to U$ be a maximal causal curve parametrized by Euclidean arclength.
 If $|\gamma ' (t)| _g >0$ for all $t\in I$, then $\gamma$ is $\mathcal C^{1,1}$.
\end{thm}

\begin{proof} The statement is local and we may assume that $I$ is compact. By continuity 
	of $\gamma'$, Proposition \ref{prop:half_H_cont}, we find some $e>0$ such that $|\gamma'(t)|_g >2\cdot e$, for all $t\in I$. Restricting $\gamma$ to a sufficiently small
	subinterval, we see that all linear segments between points on $\gamma$ are timelike curves.

	Assume that $\gamma$ is not $\mathcal C^{1,1}$. Relying on Lemma \ref{lem:geo_diff_crit} and arguing as in the proof of Proposition \ref{prop:half_H_cont}, we arrive at the following situation. For an arbitrarily large $C>0$ we find an arbitrarily small $h>0$ with the following property.  The restriction of $\gamma $ to an interval $[a,a+h]\subset I$ contains a point 
	at distance 
	\begin{equation} \label{eq: D}
 	D\geq C\cdot h^2
 	\end{equation}
	from the linear segment $\eta$ connecting $\gamma (a)$ and $\gamma (a+h)$.  We may assume without loss of generality that $a=0$ and $\gamma(0)=0 \in U$.

	Denote by $u$ the Euclidean unit vector
	$$u:= \frac {\gamma (h)- \gamma (0)} {\|\gamma (h)- \gamma (0)\|}$$
	and let the hyperplane $H$ denote the $g_0$-orthogonal complement of $u$ in $\R^n$.
	Since $u$ is timelike, the hyperplane $H$ consists of $g_0$-spacelike vectors.
	Moreover, by compactness, we find some constant $C_1>0$, depending only on $e$, such that
	\begin{equation*}
	-g_0 (w,w) \leq C_1 \cdot \|w\|^2
	\end{equation*}
	for all $w\in H$.  By continuity of $g$, we may assume, adjusting $C_1$, that also for all $x \in \R^n$  with $\|x\| \leq h$ and all $w\in H$  we have the inequality 
\begin{equation}  \label{eq: third}
-g_x (w,w) \leq C_1 \cdot \|w\|^2
\end{equation}

	
 Consider the unique parametrization $\eta :[0,h] \to U$ of the linear segment $\eta$ such that
 $$v(t):= \gamma (t)-\eta (t)$$ 
 is $g_0$-orthogonal  to $u$.  In other words, $\eta (t)$ is the $g_0$-orthogonal projection of $\gamma (t)$ onto the line spanned by $u$.  
 
 The curves $\eta$ and $v$ are well-defined and $\mathcal C^{1,\frac 1 4}$.  Moreover,
 $v(0)=v(h)=0$. By the uniform continuity of $v'$, for an arbitrarily small $\rho >0$, the inequality  
  $$\|v'(t)\| +\|v(t)\| <\rho$$ 
 holds true for all $t\in [0,h]$ once $h$ has been chosen small enough, compare \cite[Lemma 2.1]{Yaman}.
	
For sufficiently small $h$ (and thus, sufficiently small $\rho$), this implies

 \begin{equation}\label{eq:18}
 \frac 1 2 < \left\|\eta'(t)\right\|<2  \; \; \;   \text{and} \; \; \;
  |\eta ' (t)| _{\gamma (t)} >e\;.
  \end{equation}

 Now we can  estimate for all $h$ small enough:
\begin{equation*}\label{eq:19}
\begin{split}
&\left|\eta'(t)\right|_{\eta(t)}-|\gamma'(t)|_{\gamma(t)} \\
 &= \left|\eta'(t)\right|_{\eta(t)} - \sqrt{\left|\eta'(t)\right|^2_{\gamma(t)}+2g_{\gamma(t)}(\eta'(t),v'(t))+g_{\gamma(t)}(v'(t),v'(t))} \\
											&\geq \left|\eta'(t)\right|_{\eta(t)} - \left|\eta'(t)\right|_{\gamma(t)} -\frac{g_{\gamma(t)}(\eta'(t),v'(t))}{\left|\eta'(t)\right|_{\gamma(t)}}-\frac{g_{\gamma(t)}(v'(t),v'(t))}{2\left|\eta'(t)\right|_{\gamma(t)}} \\
											&\geq - 2L \cdot \left\|v(t)\right\|-\frac{\left|g_{\gamma(t)}(\eta'(t),v'(t))\right|}{e}-\frac{g_{\gamma(t)}(v'(t),v'(t))}{2e}. \\
\end{split}
\end{equation*}
Here we have used  $\sqrt{1+x}\leq1+\frac x 2$ for $x>-1$ in the first inequality and the $L$-Lipschitz continuity of  $g$ in the second inequality. 

We  estimate the second and third summand separately.
In order to deal with the second term, we observe that  for a unit vector $w \in H$:
\begin{equation*}\label{eq:21}
	\left|g_{\gamma(t)}(u,w)\right|\leq L \cdot  \|\gamma(t) \| \leq L\cdot t
\end{equation*}
due to  $g_0(u,w)=0$, the  Lipschitz continuity of $g$ and the $1$-Lipschitz continuity of $\gamma$. Therefore,
\begin{equation*}
	- \left|g_{\gamma(t)}(\eta'(t),v'(t))\right| \geq - {2 L} \cdot t  \cdot  \left\|v'(t)\right\|.
\end{equation*}

Estimating the third summand by \eqref{eq: third} we find


\begin{equation*}\label{eq:22}
\left|\eta'(t)\right|_{\eta(t)}-|\gamma'(t)|_{\gamma(t)} \geq - {2L} \left\|v(t)\right\| - \frac {2L} {e} \cdot t  \left\|v'(t)\right\| + \frac {C_1} {2e}  \left\|v'(t)\right\|^2
\end{equation*}
 Now we integrate over $[0,h]$ and use the maximality of $L(\gamma)$ to deduce that the integral of the right hand side form $0$ to $h$ is non-positive. Thus, for a constant $C_2 >0$ (depending only on $e$ and $L$) 
\begin{equation}\label{eq:23}
\int_{0}^h \left\|v'(t)\right\|^2 dt \leq C_ 2\int_{0}^h (\left\|v(t)\right\|+ \left\|v'(t)\right\| t ) \, dt \;.
\end{equation}


Since $v(0)=v(h)=0$ we use integration by parts to deduce
$$\int _0 ^h \|v(t)\|  \, dt = - \int _0 ^h (\| v(t)\| ) ' \cdot t \, dt \;.$$
Thus, using $|(\|v(t)\|) '| \leq \|v'(t)\|$ we obtain
$$ \int _0 ^h \|v'(t)\|^2 \, dt \leq 2 C_2 \cdot \int _0 ^h \|v'(t)\| \cdot t \, dt \;.$$
We apply Cauchy-Schwarz to deduce
$$\left(\int_{0}^h \left\|v'\right\|^2 \right)^2 \leq 4 C_2 ^2 \left(\int_{0}^h \left\|v'\right\|^2 \right)\cdot 
\left(\int_{0}^h t^2\, dt\right) = \frac 4 3 C_2^2 \cdot h^3 \cdot   \int_{0}^h \left\|v'\right\|^2 \;.$$
Thus, we obtain
$$\frac 4 3 C_2 ^2 h^3 \geq \int_{0}^h \left\|v'\right\|^2 \geq  \frac 1 h \cdot \left(\int _0 ^h \|v'\| \right)^2 \geq \frac 1 h \cdot v_0^2 \;,$$
with $v_0=\max_{t\in[0,h]}\|v(t)\|$. Thus, $\|v(t)\| \leq  2C_2 \cdot h^2$
for all $t\in [0,h]$. 

This provides a contradiction with \eqref{eq: D} and finishes the proof  of Theorem \ref{thm:extremal_timelike}.
\end{proof}

\section{Timelike \extremal curves}\label{sec:TL_ex_c}
The following result 
 proves and strengthens   Proposition \ref{prop: graf}.
\begin{prop}  \label{prop: final}
 Let $\gamma:[0,r] \to U$ be a \extremal curve parametrized by Euclidean arclength. If $\mathcal L(\gamma)=0$ then $|\gamma '(t)|_g =0$ for all $t$.
 
 If $\mathcal L(\gamma) >0$ then 
 $\gamma$ is $\mathcal C^{1,1}$.  Moreover, for all $t\in  [0,r]$, 
 \begin{equation} \label{eq: velcont} 
 |\gamma '(t)|_g  \geq K\cdot \frac { \mathcal L(\gamma)}  {r} >0 \;,
 \end{equation}
 where the positive constant $K$ depends only on $U$. The curve $\gamma$ admits parametrizations with constant Lorentzian velocity and any such parametrization solves the geodesic equation in the sense of Filippov.
 %
\end{prop}

\begin{proof}
By Proposition \ref{prop:half_H_cont}, $\gamma$  is $\mathcal C^{1}$. If $\mathcal L(\gamma) =0$
then $|\gamma '(t)| _g =0$, for almost all $t\in [0,r]$. By continuity of $\gamma'$, this equality holds for all  $t$.

Assume $\mathcal L (\gamma)>0$.  Let $[a_0,b_0] \subset [0,r]$ be such that
 $|\gamma '(t)| _g >0$, for all $t\in [a_0,b_0]$.  By Theorem \ref{thm:extremal_timelike}, the restriction
 $\gamma :[a_0,b_0] \to U$ is $\mathcal C^{1,1}$.   
 
  For any interval $[a,b]\subset \R$ and  a non-decreasing Lipschitz function $f:[a,b] \to [a_0,b_0]$ the  reparametrization 
  $$\tilde \gamma =\gamma\circ f : [a,b] \to U$$
  has constant Lorentzian velocity  $\ell >0$ if and only if 
   \begin{equation} \label{eq: repar}
   f'(t)= \frac {\ell} {|\gamma ' (f(t))| _g} \,,
   \end{equation}
  for almost all $t \in [a,b]$. This is a differential equation for $f$ which has a unique maximal solution with $f(a)=a_0$. Integrating we see that $f(b_0)=b$ if and only if 
  $$\ell \cdot (b-a)=\mathcal L(\gamma |_{a_0,b_0]}) \;.$$
 
  The function $f'(t)$ is Lipschitz. Hence $f$ and, therefore, $\tilde \gamma$ are $\mathcal C^{1,1}$.
  
  Due to Proposition \ref{prop: geodeq}, $\tilde \gamma$ solves the geodesic equation in the sense of Filippov. Applying Lemma \ref{lem: fil}, for all $t\in [a,b]$,
  $$f'(t) =\|\tilde \gamma '(t)\| \leq  \frac {e^{Cr_0}-1} {C}   \cdot \frac 1 {b-a}\;,$$
  where the constant $C=C(U)>0$ depends only on the Lipschitz constant of $g$ and where 
  $r_0 =b_0-a_0$ denotes the Euclidean length of $\gamma |_{[a_0,b_0]}$. Inserting into \eqref{eq: repar}, we observe for all $t_0\in [a_0,b_0]$
  $$|\gamma '(t_0)| _g \geq  \ell \cdot (b-a) \cdot \frac {C} {e^{Cr_0} -1}  = \frac {Cr_0} {e^{Cr_0} -1} \cdot \frac 1 {r_0} \cdot \mathcal L (\gamma |_{[a_0,b_0]}) \;. $$
 
  Note that $ \frac {Cr_0} {e^{Cr_0} -1}$ is decreasing in $r_0$.   Thus,  for all $t_0\in [a_0,b_0]$,
  \begin{equation} \label{eq: subint}
  |\gamma '(t_0)|_g  \geq K\cdot \frac  {\mathcal L (\gamma |_{[a_0,b_0]})} {r_0} \;,
  \end{equation}
  where the constant $K=K(U)$ denotes $K=  \frac {Cr_{max}} {e^{Cr_{max}} -1}$ and $r_{max}$ is the maximal Euclidean length of causal curves in $U$ (which is bounded by twice the diameter of $U$).
  
  It remains to prove that the set 
  $T$ of all times $t\in [0,r]$ with $|\gamma '(t)|_g=0$ is empty.  Assuming the contrary and using $\mathcal L(\gamma)>0$,  we find a subinterval $[a_1,b_1] \subset  [0,r]$ such that $|\gamma '(t)| _g$ is positive for all 
  $t\in (a_1,b_1)$ and vanishes at one of the boundary points $a_1$ or $b_1$.
  
  We  let $[a_0,b_0] \subset (a_1,b_1)$ converge to $[a_1,b_1]$. Applying \eqref{eq: subint} to the intervals $[a_0,b_0]$, we   find a uniform positive  lower bound on $|\gamma '(t)|_g$ for all
  $t \in (a_1,b_1)$. By continuity, this positive bound is valid at the boundary points $a_1$ and $b_1$ as well. This contradiction finishes the proof. 
 \end{proof}  
  
   \begin{rem}
   	As the proof shows the constant $K$ can be chosen arbitrarily close to $1$, if $U$ is chosen sufficiently small.
   \end{rem}

\section{Final arguments}\label{sec:fin}
Now we are in position to  finish the proof of Theorem \ref{thm: main}.  The remaining control 
of lightlike curves is obtained by a limiting argument, independently discovered in \cite{Schinnerl:20}:

\begin{proof}[Proof of Theorem \ref{thm: main}]
  All statements of the theorem are local. Thus we may restrict to a sufficiently small chart $U$
  satisfying the assumptions of Section \ref{sec:not}.

  We call a causal curve $\gamma$ in $U$ tame if $\gamma$ is maximal and admits a $\mathcal C^{1,1}$-parametrization solving the geodesic equation in the sense of Filippov.   We need to prove that all maximal curves are tame. Due to Proposition \ref{prop: final}, all maximal curves of positive Lorentzian length are tame.

  We first claim  that any pair $x,y \in U$ of causally related points is connected by a tame curve. Without loss of generality let $y$ be in the future of $x$.  We find a sequence $y_n $ converging to $y$, such that $y$ and $y_n$ are related by a future directed timelike curve.  Then $x$ and $y_n$ are related by future directed causal curves of positive Lorentzian length. By global hyperbolicity of $U$ we find a maximal curve $\gamma _n$ in $U$ connecting $x$ and $y_n$.  The curve $\gamma _n$ satisfies $\mathcal L(\gamma _n)>0$.   By Proposition \ref{prop: final}, we can parametrize $\gamma_n$ on the interval $[0,1]$ with constant Lorentzian speed, so that it solves the geodesic equation in the sense of Filippov.
  
  By global hyperbolicity of $U$ and the uniform bound on velocities provided by Lemma \ref{lem: fil},
  the curves $\gamma _n$ converge pointwise, after choosing a subsequence, to a future directed causal curve $\gamma:[0,1] \to U$ connecting $x$ and $y$. Due to Corollary \ref{cor_fil_lim}, this limiting curve $\gamma$ is tame.

  Let now  $\gamma$ be an arbitrary maximal curve in  $U$. When parametrized by Euclidean arclength, 
  the curve $\gamma:[0,r]\to U$ is $\mathcal C^1$ by Proposition \ref{prop:half_H_cont}.  If $\mathcal L(\gamma) >0$ then $\gamma$ is tame by Proposition \ref{prop: final}.
  
  Thus we may and will assume $L(\gamma)=0$.  For any natural $m>1$ and $0\leq i\leq m$, consider points
  $x_i= \gamma (\frac i m \cdot r) $ subdividing $\gamma$ into $m$ pieces of equal length.
  By the above, we find for any $i=0,...,m-1$ a tame curve $\gamma ^{i}_m$ connecting $x_i$ and $x_{i+1}$.
  Consider the concatenation $\gamma _m$ of the curves $\gamma ^{i}_m$ for $i=0,...,m-1$.
  The curve $\gamma _m$ is a causal curve connecting $x=\gamma (0)$ and $y=\gamma (r)$.
  
  Clearly, for any parametrizations of the curves $\gamma _m$ with uniformly bounded speed, 
  the curves subconverge pointwise to a reparametrization of $\gamma$.   Hence, due to Lemma \ref{lem: fil} and Corollary \ref{cor_fil_lim}, in order to prove  that $\gamma $ is tame, we only need to verify that all curves $\gamma_m$ are tame.

   Since $\gamma$ is maximal and $\mathcal L(\gamma)=0$, any causal curve connecting $x$ and $y$ is maximal. In particular, 
  $\gamma  _m$ is maximal. By Proposition \ref{prop:half_H_cont},  $\gamma_m$ is $\mathcal C^1$ when parametrized by Euclidean arclength. By construction, all  $\gamma^{i}_m$ admit parametrizations solving the geodesic equation in the sense of Filippov.  
  
  Any affine reparametrization of a $\mathcal C^{1,1}$-curve solving the geodesic equation in the sense of Filippov is again a solution of the geodesic equation. Due to Lemma \ref{lem: fil},
  we can find such parametrization, having any prescribed positive Euclidean velocity at the starting point. 
  
   We now  start with any 
  parametrization $\gamma ^0_m :[0,t_1] \to U$ solving the geodesic equation and proceed by induction on $i$, to find a  parametrization $\gamma ^i_m :[t_i, t_{i+1}] \to U$ solving the geodesic equation and such that $\|(\gamma^i_m )'(t_i)\| =\|(\gamma ^{i-1} _m )'(t_i)\|$.
  
  The arising concatenation $\tilde \gamma _m :[0,t_{m}] \to U$ is a reparametrization of the $\mathcal C^1$-curve  $\gamma _m$ with the following properties.  The restriction of 
  $\tilde \gamma _m$ to any of the subintervals $[t_i,t_{i+1}]$ is $\mathcal C^{1,1}$ and solves
  on this subinterval  the geodesic equation in the sense of Filippov.   
  At the boundary points $t_i$ the incoming and the outgoing directions  of $\tilde \gamma_m$ have the same norms. Moreover, since $\gamma _m$ is $\mathcal C^1$, these incoming
  and outgoing directions are parallel vectors, hence they coincide.
  
  Therefore, the derivative  $\tilde \gamma _m '$ is continuous. Since $\tilde \gamma _m'$ is Lipschitz continuous on the subintervals $[t_i,t_{i+1}]$, the derivative $\tilde \gamma _m'$
  is Lipschitz continuous on all of $[0,t_m]$. Thus, $\tilde \gamma _m$ is a $\mathcal C^{1,1}$-curve.  Clearly, $\tilde \gamma _m$  solves  the geodesic equation almost everywhere 
  on  any interval $[t_i,t_{i+1}]$, hence almost everywhere on the whole interval of definition
  	$[0,t_m]$.

  	This shows the tameness of   $\gamma _m$  and finishes the proof.
  \end{proof}
  

\bibliographystyle{alpha}
\bibliography{Lorentz}

\end{document}